\documentclass
[
    a4paper,
    DIV=11,
    abstract=true,
    11pt,
]
{scrartcl}

\usepackage
{
    amsmath,
    amssymb,
    amsthm,
    authblk,
    dsfont, % for 1 vector or indicator function
    enumitem,
    graphicx,
    mathtools,
    nicefrac,
    xcolor,
}

\usepackage{algpseudocode}
\usepackage{algorithm}

\usepackage[position=top]{subfig}

\usepackage[pdffitwindow=false,
            plainpages=false,
            pdfpagelabels=true,
            pdfpagemode=UseOutlines,
            pdfpagelayout=SinglePage,
            bookmarks=false,
            colorlinks=true,
            hyperfootnotes=false,
            linkcolor=blue,
            urlcolor=blue!30!black,
            citecolor=green!50!black]{hyperref}

\usepackage[labelfont=bf, justification=justified,format=plain]{caption}

\graphicspath{{pics/}}

\newcommand{\R}{\mathbb{R}}                                      % real numbers
\newcommand{\ts}{\hspace*{0.1em}}                                % thin space
\DeclareMathOperator{\tr}{tr}

\DeclareMathOperator*{\argmin}{argmin}

\newtheorem{theorem}{Theorem}[section]

\newtheorem{proposition}[theorem]{Proposition}

\theoremstyle{definition}

\newtheorem{remark}[theorem]{Remark}

\makeatletter
\renewcommand*\env@matrix[1][*\c@MaxMatrixCols c]{%
  \hskip -\arraycolsep
  \let\@ifnextchar\new@ifnextchar
  \array{#1}}
\makeatother

\mathtoolsset{centercolon} % for := and =:

\allowdisplaybreaks

\DeclareCaptionLabelFormat{period}{#1~#2}
\makeatletter
\def\blfootnote{\gdef\@thefnmark{}\@footnotetext}
\makeatother

\begin{document}

\title{Learning dynamical systems from data: Gradient-based dictionary optimization}
\author[1]{Mohammad Tabish\thanks{Corresponding author: \href{mailto:M.Tabish-1@sms.ed.ac.uk}{M.Tabish-1@sms.ed.ac.uk}}}
\author[2]{Neil K. Chada}
\author[3]{Stefan Klus}
\affil[1]{Maxwell Institute for Mathematical Sciences, University of Edinburgh and Heriot--Watt University, Edinburgh, UK}
\affil[2]{Department of Mathematics, City University of Hong Kong, Hong Kong SAR}
\affil[3]{School of Mathematical and Computer Sciences, Heriot--Watt University, Edinburgh, UK}
\date{}
\maketitle

\begin{abstract}
The Koopman operator plays a crucial role in analyzing the global behavior of dynamical systems. Existing data-driven methods for approximating the Koopman operator or discovering the governing equations of the underlying system typically require a fixed set of basis functions, also called \emph{dictionary}. The optimal choice of basis functions is highly problem-dependent and often requires domain knowledge. We present a novel gradient descent-based optimization framework for learning suitable and interpretable basis functions from data and show how it can be used in combination with EDMD, SINDy, and PDE-FIND. We illustrate the efficacy of the proposed approach with the aid of various benchmark problems such as the Ornstein--Uhlenbeck process, Chua's circuit, a nonlinear heat equation, as well as protein-folding data. \\[1ex] { \footnotesize\textbf{Keywords}: Koopman operator, system identification, dictionary learning, gradient descent }
\end{abstract}

\section{Introduction}

Dynamical systems can be used to describe the motion of atoms, fluids, and planets as well as biological and chemical processes to name just a few examples. Deriving mathematical models for such complex problems can be challenging. Even if we do have mathematical models, the resulting dynamical systems will often be high-dimensional and highly nonlinear, which makes their analysis difficult or sometimes impossible. The goal of data-driven modeling approaches is to learn the governing equations or transfer operators associated with the system from measurement data. Instead of analyzing individual trajectories of the system, transfer operators such as the Koopman operator and Perron--Frobenius operator describe the evolution of observables and probability densities \cite{LaMa94, DJ99, Mezic05, KKS16}. Data-driven methods allow us to study the global behavior of the system without requiring detailed mathematical models, see \cite{KD24} for an overview of different applications. Of particular interest are the eigenvalues and eigenfunctions of transfer operators since they contain important information about timescales and slowly evolving spatiotemporal patterns of the systems. Often only a few dominant eigenfunctions and modes can help us understand the dynamics. Among the most frequently used techniques for numerically approximating transfer operators are \emph{Ulam's method} \cite{Ulam60, CU02}, \emph{extended dynamic mode decomposition} (EDMD)~\cite{williams2015data, KKS16}, and various extensions such as \emph{kernel EDMD} (kEDMD)~\cite{WRK15, KSM20} and \emph{generator EDMD} (gEDMD)~\cite{KNPNCS20}. While standard EDMD requires choosing a finite set of basis functions, kernel EDMD maps the data to a potentially infinite-dimensional feature space implicitly defined by the kernel. In the infinite-data limit, EDMD converges to a Galerkin projection of the Koopman operator onto the space spanned by the set of basis functions. Ulam's method can be regarded as a special where the dictionary contains indicator functions for a decomposition of the domain into disjoint sets \cite{KKS16}.

Although the Koopman operator and its generator can also be used for system identification and forecasting, a method that directly identifies the governing equations, called \emph{sparse identification of nonlinear dynamics} (SINDy), was proposed in \cite{brunton2016discovering}. Relationships between the Koopman generator and system identification are discussed in more detail in~\cite{MauGon16, KNPNCS20}. Just like the methods for approximating transfer operators, SINDy also requires a set of basis functions. The approach can also be extended to learn partial differential equations from data. The resulting method is called \emph{PDE functional identification of nonlinear dynamics} (PDE-FIND) \cite{rudy2017data}.

The accuracy of the learned transfer operators or governing equations depends strongly on the dictionary. Poorly chosen basis functions can result in ill-conditioned matrices, spectral pollution, and incorrect predictions. Selecting suitable basis functions is an open problem. In practice, often dictionaries comprising, for instance, monomials, trigonometric functions, or radial-basis functions are used. Our goal is to develop a more flexible framework for data-driven modeling approaches that not only learns the dynamics but also, at the same time, the basis functions. Although neural network-based dictionary learning methods have been successfully used for approximating transfer operators in combination with EDMD and to identify governing equations, see, e.g., \cite{mardt2018vampnets, gulina2021two, li2017extended, jin2024extended, enoch2019}, where the output layer of the network represents the basis functions, we lose the interpretability of the dictionary and cannot immediately identify the governing equations. We propose a dictionary learning approach that is based on well-known gradient-based optimization techniques such as stochastic gradient descent~(SGD), Nesterov's method, or Adam, see \cite{kingma2014adam, lu2022gradient, nesterov2018lectures}, and allows us to optimize the parameters of basis functions (e.g., the centers and bandwidths of Gaussians). The key advantage of our method is that it is a generalized framework that encapsulates different data-driven algorithms for learning dynamical systems and results in an interpretable representation. It can be regarded as a compromise between the flexibility of fully optimizable basis functions generated by a neural network and the interpretability and expressivity of pre-selected dictionaries. Alternating optimization algorithms to learn the Koopman operator have also been proposed in \cite{liu2020towards, enoch2019} using the reconstruction error as a loss function. However, we present a general framework for alternating optimization to learn the dynamics and optimal and interpretable dictionaries using different loss functions. The main contributions of this work are:
\begin{enumerate}
\item We propose a novel framework for learning dynamical systems from data using gradient-based optimization methods, which allow us to simultaneously learn the dynamics and suitable basis functions.
\item We show how this approach can be used to approximate the Koopman operator and to learn governing equations and illustrate its advantages over existing techniques.
\item We demonstrate the results with the aid of several different applications ranging from chaotic dynamical systems and protein-folding problems to a nonlinear heat equation.
\end{enumerate}
The remainder of the paper is structured as follows: We first introduce the required concepts including transfer operators, EDMD, and SINDy as well as gradient-based optimization algorithms in Section~\ref{sec:background}. In Section~\ref{sec:proposed_algos}, we discuss the proposed optimization framework for learning dynamical systems from data. We present numerical results for various deterministic and stochastic dynamical systems in Section~\ref{sec:numerical}. Open problems and a discussion of future work can be found in Section~\ref{sec:conclusion}.

\section{Background}
\label{sec:background}

In this section, we provide the necessary background material required for the derivation of our optimization framework, including the Koopman operator, data-driven algorithms for learning dynamical systems, as well as gradient-based optimization techniques.

\subsection{Koopman operator}

Let $\mathbb{X} \subset \R^d$ be the state space and $S^t \colon \mathbb{X} \rightarrow \mathbb{X}$ the flow map associated with a given dynamical system, i.e., for an initial condition $ x(0) = x_0 $, it holds that $ S^t(x_0) = x(t)$. We are in particular interested in stochastic processes $ \{X_t\}_{t \geq 0} $ governed by \emph{stochastic differential equations} (SDE) of the form
\begin{equation*}
    \mathrm{d}X_t = b(X_t) \ts \mathrm{d}t + \sigma(X_t) \ts \mathrm{d}W_t,
\end{equation*}
where $ b \colon \R^d \rightarrow \R^d $ is the drift term, $\sigma \colon \R^d \rightarrow \R^{d \times d}$ is the diffusion term, and $ W_t $ is a $ d $-dimensional Wiener process. Let $ \mathbb{E}[\cdot] $ denote the expectation of a random variable. The semigroup of Koopman operators $ \{\mathcal{K}^t\}_{t \ge 0} $, with $\mathcal{K}^t \colon \mathcal{L}^\infty \rightarrow \mathcal{L}^\infty$, is defined by
\begin{align*}
    (\mathcal{K}^tf)(x) = \mathbb{E}[f(S^t(x))].
\end{align*}
The Koopman operator describes the evolution of observables rather than the evolution of the state of the system, see \cite{LaMa94, Mezic05, KKS16, hollingsworth2008stochastic} for more details.

\subsection{Extended dynamic mode decomposition}

A data-driven algorithm for approximating the Koopman operator, called \emph{extended dynamic mode decomposition} (EDMD), was proposed in \cite{williams2015data}. EDMD requires trajectory data $ \{(x_i, y_i)\}_{i=1}^m $, where $ y_i = S^\tau(x_i) $ for a fixed lag time $ \tau $. We select a set of basis functions $\mathcal{D} = \{\psi_1, \psi_2, \dots, \psi_n \} $, also known as \emph{dictionary}. These basis functions can be, for instance, monomials, indicator functions, or radial basis functions. We define the vector-valued function $ \psi \colon \R^d \rightarrow \R^n $ by
\begin{equation*}
    \psi(x) = [\psi_1(x), \psi_2(x), \dots, \psi_n(x)]^\top.
\end{equation*}
EDMD aims to find the best approximation of the Koopman operator $ \mathcal{K}^\tau $ projected onto the space spanned by the selected basis functions. In what follows, let $ \lVert\cdot\rVert_F $ denote the Frobenius norm. In order to obtain an approximation of the Koopman operator, we first map the training data to the feature space by defining
\begin{equation*}
    \Psi_x = [\psi(x_{1}),\psi(x_{2}),\ldots,\psi(x_{m})] \in \R^{n \times m}
    ~~ \text{and} ~~
    \Psi_y = [\psi(y_{1}),\psi(y_{2}),\ldots,\psi(y_{m})] \in \R^{n \times m}
\end{equation*}
and then minimize the loss function
\begin{equation*}
    \mathcal{F}(K) = \big \|\Psi_y - K^\top \Psi_x\big \|_F.
\end{equation*}
An optimal solution of the regression problem is given by
\begin{equation*}
    K^\top = \Psi_y \Psi_x^+ = \big(\Psi_y \Psi_x^\top\big) \big(\Psi_x \Psi_x^\top\big)^+ = C_{yx} \ts C_{xx}^+,
\end{equation*}
where $ ^+ $ denotes the pseudoinverse and $ C_{xx} $ and $ C_{xy} = C_{yx}^\top $ are the (uncentered) covariance and cross-covariance matrices, respectively. The matrix $ K \in \R^{n \times n} $ is the matrix representation of the projected Koopman operator $ \mathcal{K}^\tau $. We can then compute the eigenvalues $ \lambda_i $ and the eigenvectors $ v_i $ of the matrix $ K $ to obtain eigenfunctions
\begin{equation*}
    \varphi_i(x) = v_i^{\top}\psi(x)
\end{equation*}
of the projected operator.

\subsection{Sparse identification of nonlinear dynamics}

Instead of learning transfer operators associated with a given dynamical system, we can also directly approximate the governing equations from data using \emph{sparse identification of nonlinear dynamics} (SINDy) \cite{brunton2016discovering}. Assuming the governing equations comprise only a few simple terms, SINDy solves a sparse regression problem to discover parsimonious models. Although SINDy has been extended to stochastic differential equations, see \cite{BNC18, KNPNCS20}, we will restrict ourselves here to autonomous ordinary differential equations
\begin{equation*}
    \dot{x}(t) = b(x(t)).
\end{equation*}
Given training data of the form $ \{ (x_i, \dot{x}_i) \}_{i=1}^m $, which can, for example, be extracted from one long trajectory, we construct the data matrix
\begin{equation*}
    \dot{X} = [\dot{x}_1, \dot{x}_2, \dots, \dot{x}_m] \in \R^{d \times m}.
\end{equation*}
The matrix $\dot{X}$ contains the time-derivatives at the training data points, which can also be approximated numerically using finite differences. In order to find an approximation of the governing equations, we minimize the loss function
\begin{equation*}
    \mathcal{F}(\Xi) = \big \| \dot{X} - \Xi^{\top}\Psi_x \big \|_F,
\end{equation*}
where $ \Xi \in \R^{n \times d} $. The least squares solution to the above problem is
\begin{equation*}
    \Xi^\top = \dot{X}\Psi_x^+ = \big(\dot{X}\Psi_x^\top\big)\big(\Psi_x\Psi_x^\top\big)^+.
\end{equation*}
The data-driven approximation of the dynamical system is then defined by
\begin{align*}
    \dot{x} \approx \Xi^{\top}\psi(x).
\end{align*}

\subsection{SINDy for PDEs}

PDE-FIND \cite{rudy2017data} is an extension of the SINDy approach to partial differential equations. The algorithm discovers the governing equations from time-series data by solving a sparse regression problem. The general form of a PDE that the algorithm recovers is
\begin{equation*}
    u_t = N(x, u, u_x, u_{xx}, \dots),
\end{equation*}
where the subscripts denote the partial derivatives with respect to $ t $ and $ x $. Suppose that we have data from a solution of the PDE on an $n \times m$ grid, i.e., for $m$ spatial and $n$ time points. We start by constructing a discretized version of the right-hand side of the above equation and express $u(x, t)$ and its derivatives using the matrix $\Theta(U)$. Each column of the matrix $\Theta(U)$ represents a candidate term that can be present in the right-hand side of the PDE. That is, if we have $d$ candidate terms, then there will be $d$ columns. Each column represents the value of a particular candidate function on all $mn$ space-time points. We can write the discretized version of the PDE as
\begin{equation*}
    U_t = \Theta(U) \ts \xi,
\end{equation*}
where $\Theta(U) \in \R^{mn \times d}$. Each nonzero entry in $\xi$ corresponds to a term in the PDE selected from the dictionary of candidate terms. To obtain a sparse vector, the algorithm minimizes the loss function
\begin{equation*}
    \mathcal{F}(\xi) = \big\| U_t - \Theta(U)\xi \big \|_2^2 + \epsilon \ts \kappa(\Theta(U))\big\| \ts\xi\big\|_0,
\end{equation*}
where the second term is added to avoid overfitting. Here, $\epsilon>0$ is a regularization parameter, $ \kappa(\Theta(U)) $ is the condition number of the matrix $ \Theta(U) $, and $\| \xi \|_0$ is the number of nonzero entries of $\xi$.

\subsection{Gradient descent algorithms}

We now turn our attention to gradient descent algorithms, which aim to solve minimization problems of the form
\begin{equation*}
    x^* = \argmin\limits_{x \in \R^n} Q(x),
\end{equation*}
where $Q$ is called \emph{loss function}, \emph{cost function}, or \emph{objective function}. There exist many different gradient descent algorithms to find a minimizer $x^*$ of the above unconstrained optimization problem. We will briefly review some of these methods, but refer the reader to \cite{lu2022gradient, nesterov2018lectures, robbins1951stochastic, tran2024gradient} for more details.

\subsubsection{Gradient descent}

Gradient descent (GD) aims to find a minimizer by following the gradient ``downhill''. Given an initial guess $ x_0 $, GD is defined by
\begin{align*}
    x_{t+1} = x_t - h \cdot \nabla_x Q(x_t),
\end{align*}
where $h>0$ is the so-called \emph{learning rate} or step size. However, computing the full gradient of the loss function can be computationally expensive.

\subsubsection{Stochastic gradient descent}

Stochastic gradient descent (SGD) is a stochastic variant of the gradient descent algorithm. First, we rewrite the loss function as
\begin{equation*}
     Q(x) = \sum_{i=1}^m Q_i(x).
\end{equation*}
We then approximate the actual gradient $\nabla_x Q(x)$ by the gradient $\nabla_x Q_i(x)$ w.r.t.\ a single data point and compute
\begin{align*}
    x_{t+1} = x_t - h \cdot \nabla_x Q_i(x_t).
\end{align*}
For optimization problems involving a large number of data points, SGD reduces the computational costs considerably. However, the convergence can be very slow due to the approximation of the true gradient. A compromise between computational efficiency and accuracy is to compute the gradient with respect to a random subset or so-called \emph{batch} of data points. This is known as mini-batch SGD.

\subsubsection{Nesterov's method}

Nesterov's method \cite{nesterov2018lectures,nesterov1983} is an accelerated version of gradient descent that attains faster convergence. The main difference is that the updates are based on a combination of the current step and the previous step, scaled using a parameter $\beta$ that is also updated with each iteration. Typically, $\beta_t$ is computed using
\begin{align*}
    p_{t+1} &= \frac{1}{2}\left(1 + \sqrt{1 + 4p_t^2}\right), \\
    \beta_t &= \frac{p_t - 1}{p_{t+1}},
\end{align*}
where $ p_0 = 0 $. We then define
\begin{align*}
    x_{t+1} = x_t + \beta_t(x_t - x_{t-1}) - h \cdot \nabla_x Q(x_t + \beta_t(x_t - x_{t-1})).
\end{align*}

\subsubsection{Adam}

The last optimization method that we will consider is Adam \cite{kingma2014adam}, which is an extension of SGD. It is among the most popular methods for training neural networks as it is computationally efficient for large-scale optimization problems, has reasonable memory requirements, and is suitable for noisy gradients as well. The algorithm is based on the adaptive estimation of first- and second-order moments and involves the calculation of moving averages of the gradients while controlling the decay of these averages with the parameters $\beta_1$ and $\beta_2$. For a theoretical analysis of the convergence of Adam, see, e.g., \cite{fossez2022a, chen2018on}. The values of the hyperparameters are typically set to the default values $\beta_1 = 0.9$, $\beta_2 = 0.999$, and $\epsilon = 1\times 10^{-8}$. The main steps for the algorithm are:
\begin{itemize}
    \item Update biased first moment: $m_{t+1} = \beta_1 m_t + (1-\beta_1) \cdot \nabla_x Q(x_{t})$.
    \item Update biased second raw moment estimate: $v_{t+1} = \beta_2 v_t + (1-\beta_2) (\nabla_x Q(x_{t}))^2$.
    \item Compute bias-corrected first moment estimate: $\hat{m}_{t+1} = m_{t+1}/(1 - \beta_1^{t+1})$.
    \item Compute bias-corrected second raw moment estimate: $\hat{v}_{t+1} = v_{t+1}/(1 - \beta_2^{t+1})$.
    \item Compute next iterate: $x_{t+1} = x_{t} - h \cdot \hat{m}_{t+1}/(\sqrt{\hat{v}_{t+1}} + \epsilon)$.
\end{itemize}

\section{Proposed optimization framework}
\label{sec:proposed_algos}

We now derive the proposed optimization framework for learning dynamical systems from data. From the above discussion on different data-driven algorithms, we have seen that they require a set of basis functions to learn the dynamics. The selection of appropriate basis functions is not straightforward since it is problem-dependent. We use gradient descent-based algorithms to learn the dynamics and the parameters of the basis functions from data in an alternating optimization fashion.

\subsection{Parametric EDMD}
To illustrate our approach, we use the example of learning the Koopman operator with a set of $n$ parametric basis functions
\begin{equation} \label{eq:edmd_param_basis}
    \mathcal{D} = \{ \psi_1(x, w_1), \psi_2(x, w_2) ,\ldots,\psi_n(x, w_n) \},
\end{equation}
where each $w_i$ is a $p_i$-dimensional vector, i.e., each function has $p_i$ unknown parameters. We define the parameter-dependent vector-valued function
\begin{equation} \label{eq:param_basis_vectors}
    \psi(x, w) = [\psi_1(x, w_1), \psi_2(x, w_2), \dots, \psi_n(x, w_n)]^\top.
\end{equation}
For a fixed parameter vector $w$, we can transform the training data $\{x_i, y_i\}_{i=1}^m$ using \eqref{eq:param_basis_vectors} to get $\Psi_x(w), \Psi_y(w) \in \mathbb{R}^{n \times m}$. We then minimize the reconstruction error, given by
\begin{equation} \label{eq:param_edmd}
    \mathcal{F}(K, w) = \big\| \Psi_y(w) - K^{\top} \Psi_x(w)\big\|_F,
\end{equation}
using gradient descent algorithms to find $K$. However, the reconstruction error might not be a good choice for optimizing the parameters of the basis functions. Suppose, for example, we choose the basis functions to be Gaussians with variable centers and bandwidths. In that case, the bandwidths of the Gaussian functions might tend to infinity so that $ \Psi_x(w) $ and $ \Psi_y(w) $ become almost constant and choosing $ K = I $ minimizes the loss function. To avoid this problem, we utilize the \emph{variational approach for Markov processes} (VAMP) score~\cite{mardt2018vampnets}. The VAMP-2 score is defined by
\begin{equation} \label{eq:vamp2_score_parametric}
        \mathcal{\hat{R}}_2(K, w) = \big\|C_{xx}^{-\nicefrac{1}{2}}(w) \ts C_{xy}(w) \ts C_{yy}^{-\nicefrac{1}{2}}(w)\big\|_F^2,
\end{equation}
where the covariance and cross-covariance matrices are given by
\begin{equation*}
    C_{xx}(w) = \Psi_x(w) \Psi_x^{\top}(w),\quad  C_{xy}(w) = \Psi_x(w) \Psi_y^{\top}(w), \quad \text{and} \quad C_{yy}(w) = \Psi_y(w) \Psi_y^{\top}(w).
\end{equation*}
It was shown in \cite{wu2020variational} that maximizing the VAMP-2 score results in basis functions associated with the slow dynamics of the system. We use the VAMP-2 score to optimize the parameters~$w$.

\begin{proposition}\label{prop:edmd_grad}
For the loss function \eqref{eq:param_edmd}, it holds that $$\nabla_K \mathcal{F}(K, w) = 2(\Psi_x(w) \Psi_x^{\top}(w) K - \Psi_x(w) \Psi_y^{\top}(w)).$$
\end{proposition}
\begin{proof} We have
\begin{align*}
    \mathcal{F}(K, w) &= \tr\big(( \Psi_y(w) - K^{\top}\Psi_x(w))^{\top}( \Psi_y(w) - K^{\top}\Psi_x(w) )\big) \\
    &= \tr\big(\Psi_y^{\top}(w) \Psi_y(w) \big) - 2 \ts \tr \big(\Psi_y^{\top}(w) K^{\top}\Psi_x(w) \big)  + \tr\big( \Psi_x^{\top}(w) K K^{\top} \Psi_x(w)\big).
\end{align*}
Computing the derivative with respect to $K$, we have
\begin{align*}
\nabla_K \mathcal{F}(K, w) &= \big(\Psi_x(w) \Psi_x^{\top}(w) K  + \Psi_x(w) \Psi_x^{\top}(w) K \big) - 2\Psi_x(w)\Psi_y^{\top}(w) \\
&= 2 \Psi_x(w) \Psi_x^{\top}(w) K - 2\Psi_x(w) \Psi_y^{\top}(w),
\end{align*}
see, e.g., \cite{petersen2008matrix}.
\end{proof}

We use the Python library JAX \cite{jax2018github} to compute the gradient of the VAMP-2 score \eqref{eq:vamp2_score_parametric} with respect to the parameters $w$. In order to be able to apply the stochastic variant of gradient descent algorithms, we need to decompose the loss function into a sum of loss functions.

\begin{proposition}
For the loss function $\mathcal{F}(K, w) = \big \| \Psi_y(w) - K^{\top}\Psi_x(w) \big \|_F^2$, we write
\begin{equation*}
    \mathcal{F}(K, w) = \sum_{i=1}^m \mathcal{F}_i(K, w), \quad \text{with} \quad
    \mathcal{F}_i(K, w) = \big \|\psi(y_i, w) - K^{\top} \psi(x_i, w)\big \|_2^2,
\end{equation*}
where $m$ is the number of data points.
\end{proposition}
\begin{proof}
This follows immediately from the property $\|A\|_F^2 = \sum_i \|a_i\|_2^2$, where $a_i$ denotes the $i$th column of $A$.
\end{proof}
Based on the above splitting of the cost function, we can write the gradient as
\begin{equation*}
    \nabla_K \mathcal{F}(K, w) = \sum_{i=1}^m \nabla_K \mathcal{F}_i(K, w)
\end{equation*}
and then select a random subset of data points to approximate the gradient.

\begin{remark}
Using the definition of the Frobenius norm, we can further show that the optimization problem \eqref{eq:param_edmd} can also be divided into subproblems for each column of the matrix $ K $ as
\begin{equation*}
    \mathcal{F}(K, w) =\big\| \Psi_y(w) - K^{\top}\Psi_x(w) \big\|^2_F = \sum_{i=1}^n \big\| [\Psi_y^{\top}(w)]_{i} - \Psi_x^{\top}(w) [K]_{i}\big\|_2^2,
\end{equation*}
where $[\Psi_y^{\top}(w)]_{i}$ is the $i$th row of $\Psi_y(w)$ and $[K]_i$ is the $i$th column of $K$. For each $i$, the above regression problem is of the form
\begin{equation*}
    \min_{x \in \R^n} \big\|Ax - b\big\|_2^2,
\end{equation*}
with $A = \Psi_x^{\top}(w) \in \mathbb{R}^{m \times n}$ and $b = [\Psi_y^{\top}(w)]_{i} \in \mathbb{R}^m$. The application of gradient-descent techniques to regression problems of this form has been studied, for example, in \cite{nesterov2018lectures}. It has been shown that a good choice for the step size is the inverse of the largest eigenvalue of the matrix $A^{\top} A$. We thus choose the step size $h = \frac{1}{\lambda_{max}}$, where $\lambda_{max}$ is the largest eigenvalue of $\Psi_x(w) \Psi_x^{\top}(w)$.
\end{remark}

\subsection{Parametric SINDy for a system of ODEs}

The optimization framework discussed above can in the same way be applied to system identification problems. We aim to discover autonomous ordinary differential equations
\begin{equation*}
    \dot{x}(t) = b(x(t), w),
\end{equation*}
where $x(t) \in \R^d$ represents the state of the system at time $t$ and $w$ is a vector of parameters. The standard SINDy algorithm cannot detect the values of parameters $w$ since it requires a fixed set of basis functions. Assume, for instance, that the right-hand side contains $ \sin(\alpha x_i) $ or $ e^{\alpha x_i} $, then SINDy would only be able to identify the governing equations if $ \sin(\alpha x_i) $ with the correct value of $\alpha$ is contained in the dictionary. To address this issue, we select a set of parametric basis functions \eqref{eq:edmd_param_basis} and compute the matrix $\Psi_x(w) \in \mathbb{R}^{n \times m}$ using \eqref{eq:param_basis_vectors}. We then minimize the reconstruction error
\begin{equation} \label{eq:param_sindy}
    \mathcal{F}(\Xi, w) = \big\|\dot{X} - \Xi^{\top}\Psi_x(w)\big\|_F.
\end{equation}
This allows us to optimize $\Xi$ and $w$ at the same time. We again use gradient descent to minimize the loss function in an alternating fashion. The approximation of the dynamical system is then given by
\begin{align*}
    \dot{x} \approx \Xi^{\top} \psi(x, w).
\end{align*}

\begin{proposition}
For the loss function \eqref{eq:param_sindy}, the derivative of $\mathcal{F}(\Xi, w)$ with respect to $ \Xi $ is given by
\begin{equation*}
    \nabla_{\Xi} \mathcal{F}(\Xi, w) = 2(\Psi_x(w) \Psi_x^{\top}(w) \Xi - \dot{X} \Psi_x^{\top}(w)).
\end{equation*}
\end{proposition}
\begin{proof}
The proof is similar to the EDMD counterpart.
\end{proof}

For computing the gradient of \eqref{eq:param_sindy} with respect to $w$, we again use JAX.

\begin{proposition}
The cost function $\mathcal{F}(\Xi, w) = \big\|\dot{X} - \Xi^{\top}\Psi_x(w)\big\|_F^2$ can be written as
\begin{equation*}
    \mathcal{F}(\Xi, w) = \sum_{i=1}^m \mathcal{F}_i(\Xi, w),
    \quad \text{with} \quad
    \mathcal{F}_i(\Xi, w) = \big\|\dot{x_i} - \Xi^{\top} \psi(x_i, w)\big\|_2^2,
\end{equation*}
where $m$ is the number of data points.
\end{proposition}

\begin{proof}
The proof again follows from the properties of the Frobenius norm.
\end{proof}

\subsection{Parametric SINDy for PDEs}

We can also extend the framework to learn PDEs from data. Assume the PDE we aim to recover is of the form
\begin{equation*}
    u_t = N(x, u, u_x, u_{xx}, f(u, w_1), \dots),
\end{equation*}
where $N(\cdot)$ is a nonlinear function of $u(x,t)$ and its partial derivatives with respect to $x$ and $t$ and $w$ is the vector of parameters of the PDE. We create the matrix $\Theta(U(w)) \in \R^{mn \times d}$ of the basis functions as discussed in Section \ref{sec:background}. We then minimize the loss function
\begin{equation} \label{eq:param_pde_loss}
    \mathcal{F}(\xi, w) =  \big\|U_t - \Theta(U(w)) \ts \xi\big\|_2^2
\end{equation}
using the proposed framework. We omit the $\| \cdot \|_0$ penalty term due to differentiability issues.

\begin{proposition}
For fixed $w$, the derivative of the loss function \eqref{eq:param_pde_loss} with respect to $ \xi $ is given by
\begin{equation*}
    \frac{\partial}{\partial \xi} \mathcal{F}(\xi, w) = 2 \ts \Theta(U(w))^\top \Theta(U(w)) \ts \xi - 2 \ts \Theta(U(w))^\top U_t.
\end{equation*}
\end{proposition}
\begin{proof}
This follows from the proof of Proposition \ref{prop:edmd_grad}.
\end{proof}

For computing the gradient of \eqref{eq:param_pde_loss} with respect to $w$, we again use JAX.

\begin{remark}
In order to avoid overfitting, it would also be possible to add regularization terms to the various loss functions. For example, SINDy is based on the assumption that there are only a few terms governing the dynamics. We can hence include an $L^1$ regularization of the matrix $\Xi$ to obtain a sparse solution. For the PDE identification approach, a regularization based on the condition number of the matrix $\Theta(U(w))$ might be beneficial as shown in \cite{rudy2017data}.
\end{remark}

\subsection{Proposed algorithm}

We now illustrate the proposed optimization framework. Let $\mathcal{L}_1(A, w)$ and $\mathcal{L}_2(A, w)$ be two arbitrary loss functions. We define an alternating Adam algorithm, which can be used for both parametric EDMD and SINDy. Other gradient-based optimization methods can be implemented in a similar way.

\begin{algorithm}[h!]
\caption{Alternating Adam}
\begin{algorithmic}
\State \textbf{Initialization:}
\begin{itemize}
 \item Select $\beta_1 = 0.9$, $\beta_2=0.999$, $\epsilon=1 \times 10^{-8}$ (default) \& a step size $h=0.01$
 \item Select initial matrices $w_{1}$ for parameters in basis and $A_0$ for dynamics matrix.
 \item $m^a_0 = 0$ (first moment for $A$) \& $v^a_0 = 0$ (second moment for $A$)
 \item $m^w_1 = 0$ (first moment for $w$) \& $v^w_1 = 0$ (second moment for $w$)
 \end{itemize}
\While{not converged}
\State $\left\{
\begin{array}{l}
m^a_{t+1} = \beta_1 m^a_t + (1-\beta_1) \cdot \nabla_A \mathcal{L}_1(A_{t}, w_{t+1}) \\
m^w_{t+1} = \beta_1 m^w_t + (1-\beta_1) \cdot \nabla_w \mathcal{L}_2(A_t, w_{t})
\end{array}
\right. $ \Comment{Update biased first moments}
\State $\left\{
\begin{array}{l}
v^a_{t+1} = \beta_2 v^a_t + (1-\beta_2) (\nabla_A \mathcal{L}_1(A_{t}, w_{t+1}))^2 \\
v^w_{t+1} = \beta_2 v^w_t + (1-\beta_2) (\nabla_w \mathcal{L}_2(A_t, w_{t}))^2
\end{array}
\right. $ \Comment{Update biased second moments}
\State $ \left\{ \begin{array}{l}
\hat{m}_{t+1}^a = m^a_{t+1}/(1 - \beta_1^{t+1}) \\
\hat{m}_{t+1}^w = m^w_{t+1}/(1 - \beta_1^{t+1})
\end{array} \right. $ \Comment{Update bias-corrected first moments}
\State $ \left\{
\begin{array}{l}
\hat{v}_{t+1}^a = v^a_{t+1}/(1 - \beta_2^{t+1}) \\
\hat{v}_{t+1}^w = v^w_{t+1}/(1 - \beta_2^{t+1})
\end{array} \right. $ \Comment{Update bias-corrected second moments}

\State $w_{t+1} = w_t - h \cdot \hat{m}_{t+1}^w/(\sqrt{\hat{v}_{t+1}^w} + \epsilon)$  \Comment{Basis parameters update}

\State $A_{t+1} = A_{t} - h \cdot \hat{m}_{t+1}^a/(\sqrt{\hat{v}_{t+1}^a} + \epsilon)$ \Comment{Dynamics matrix update}

\EndWhile
\end{algorithmic}
\label{alg:alt_adam_general}
\end{algorithm}

To apply Algorithm~\ref{alg:alt_adam_general} to parametric EDMD, we use the loss functions $\mathcal{L}_1 = \mathcal{F}$ (reconstruction error \eqref{eq:param_edmd}) and $\mathcal{L}_2 = \mathcal{\hat{R}}_2$ (VAMP-2 score \eqref{eq:vamp2_score_parametric}) for the optimization of the matrix $A = K$ and the parameters $w$, respectively. For parametric SINDy, $\mathcal{L}_1 = \mathcal{L}_2 = \mathcal{F}$ are given by the reconstruction error \eqref{eq:param_sindy}, which optimizes both the matrix $A = \Xi$ and the parameters~$w$. Similarly, for PDE-FIND, the loss function for the optimization of the vector $A = \xi$ and the parameters $w$ is the reconstruction error \eqref{eq:param_pde_loss}.

\section{Numerical results}
\label{sec:numerical}

We now apply the proposed framework to various benchmark problems.

\subsection{Parametric EDMD}

\subsubsection{Protein folding}

Analyzing the structure of proteins, especially transitions between folded and unfolded states, helps us understand biological processes and develop new drugs or treatments~\cite{dill2008protein}. We consider the protein Chignolin (CLN025) consisting of $10$ residues and $166$ atoms. The simulation data was generated by \href{https://www.deshawresearch.com/}{D.E. Shaw Research}, see \cite{lindorff2011fast} for more details. The trajectory is $1.06 \cdot 10^8$\ts ps long. We subsample the trajectory to create the training data $\{x_i, y_i\}_{i=1}^m$ using the lag time $\tau=9911.5$\ts ps. We then transform the trajectory data using contact maps, which measure the distances between different pairs of residues, resulting in $\Psi_x, \Psi_y \in \R^{28 \times 10,694}$. Since the basis functions are in this case the contact distances, there are no parameters to optimize.

To approximate the Koopman operator, we apply standard EDMD and the proposed optimization framework. The relative approximation error for different gradient descent-based algorithms compared to the exact EDMD solution are shown in Figure \ref{subfig:error_chig}. Eigenvalues of the approximated Koopman operator are given in Figure \ref{subfig:chig_eigv}. There is a spectral gap between the second and third eigenvalue, which indicates that there exist two metastable states. Furthermore, the values of the two dominant eigenfunctions, shown in Figure \ref{subfig:chig_eigf_oa}, can be clustered into two sets representing the folded and unfolded states of the molecule. Examples of folded and unfolded states are shown in Figures~\ref{subfig:chig_folded_states} and \ref{subfig:chig_unfolded_states}, respectively. The frequency of contacts between different residues of the molecule for the set of folded and unfolded states is presented in Figures \ref{subfig:chig_folded_contact} and \ref{subfig:chig_unfolded_contact}, respectively. We can see that Chignolin is mostly in its folded (native) state. The example shows that the gradient descent algorithms converge to the EDMD approximation and that Adam eventually outperforms the other gradient descent approaches.

\begin{figure}
    \centering
    \subfloat[][\label{subfig:error_chig}]{\includegraphics[width=.31\textwidth]{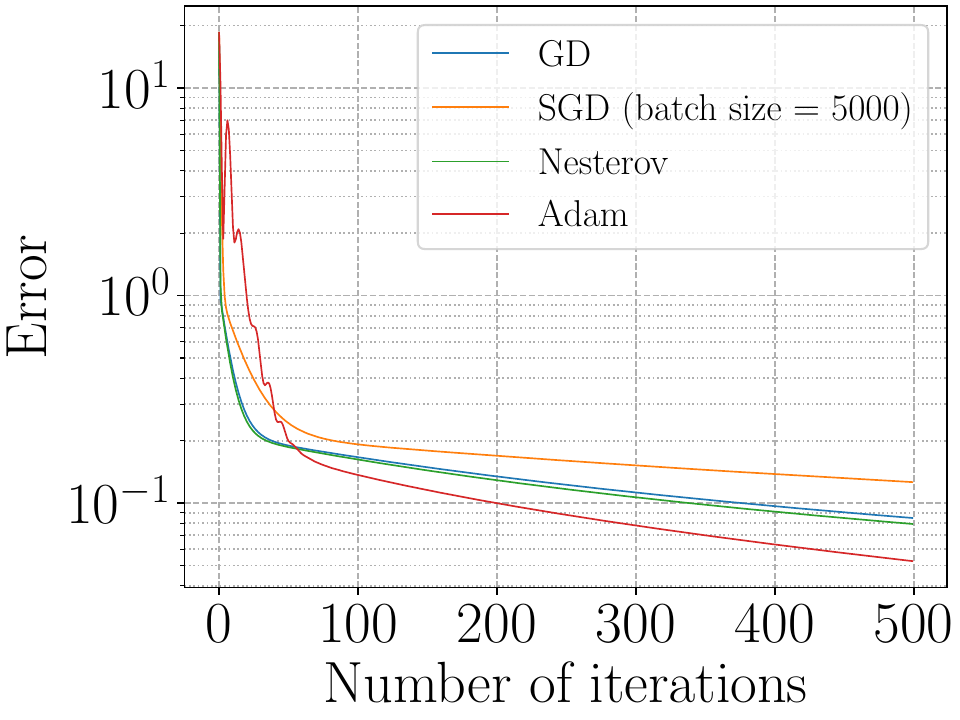}}\quad
    \subfloat[][\label{subfig:chig_eigv}]{\includegraphics[width=.29\textwidth]{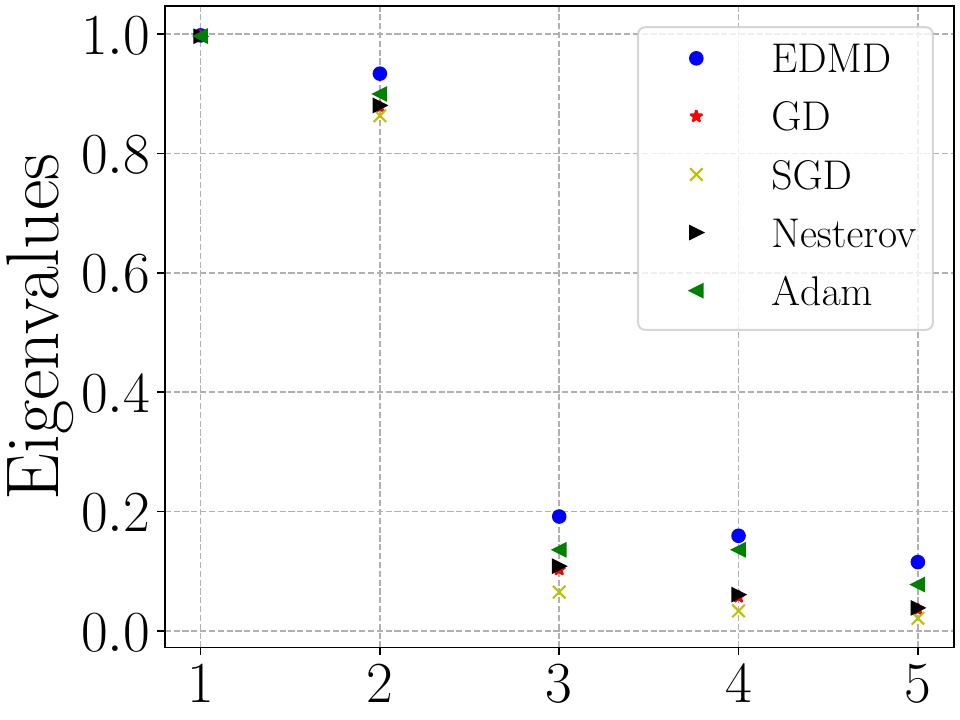}}\quad
    \subfloat[][\label{subfig:chig_eigf_oa}]{\includegraphics[width=.31\textwidth]{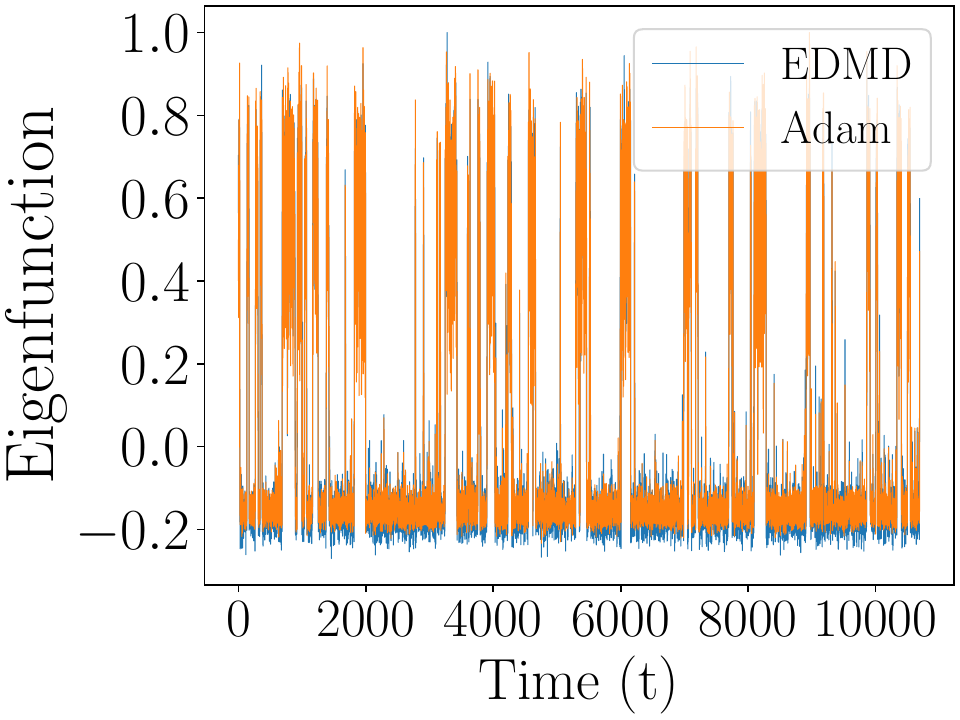}}\quad
    \caption{Koopman operator approximation for the Chignolin protein. (a) Convergence of the relative approximation error of different algorithms for the matrix $K$. (b) Eigenvalues of the approximated Koopman operator. (c) Values of the second eigenfunction of the approximated Koopman operator.}
    \label{fig:chig_results}
\end{figure}

\begin{figure}
    \centering
    \begin{minipage}[t]{0.4\linewidth}
        \subfloat[][\label{subfig:chig_folded_states}]{\includegraphics[width=0.8\textwidth]{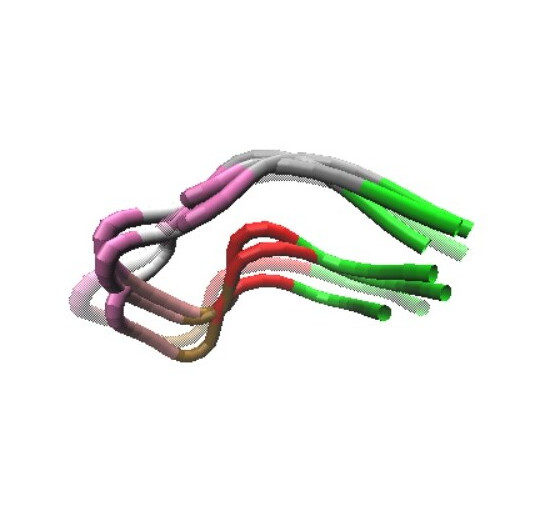}
    }
    \end{minipage}
    \begin{minipage}[t]{0.4\linewidth}
        \subfloat[][\label{subfig:chig_unfolded_states}]{\includegraphics[width=0.7\textwidth]{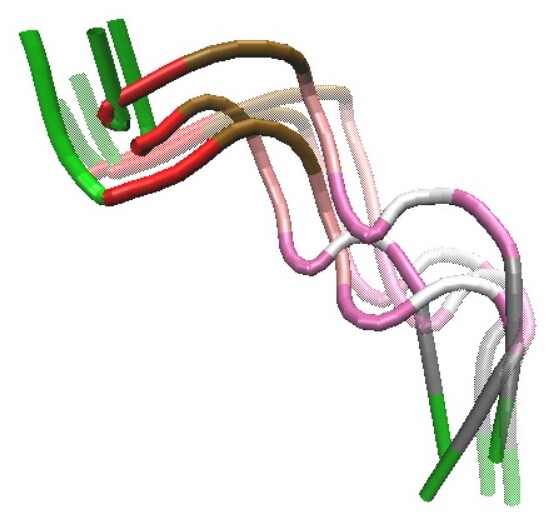}}
    \end{minipage} \\
    \begin{minipage}[t]{0.4\linewidth}
        \subfloat[][\label{subfig:chig_folded_contact}]{\includegraphics[width=0.9\textwidth]{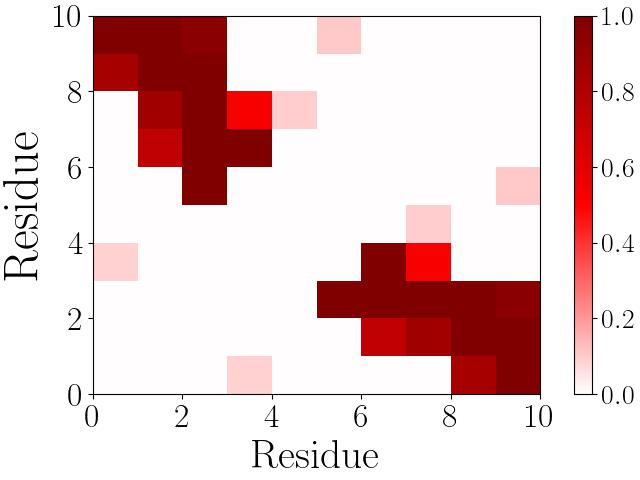}}
    \end{minipage}
    \begin{minipage}[t]{0.4\linewidth}
        \subfloat[][\label{subfig:chig_unfolded_contact}]{\includegraphics[width=0.9\textwidth]{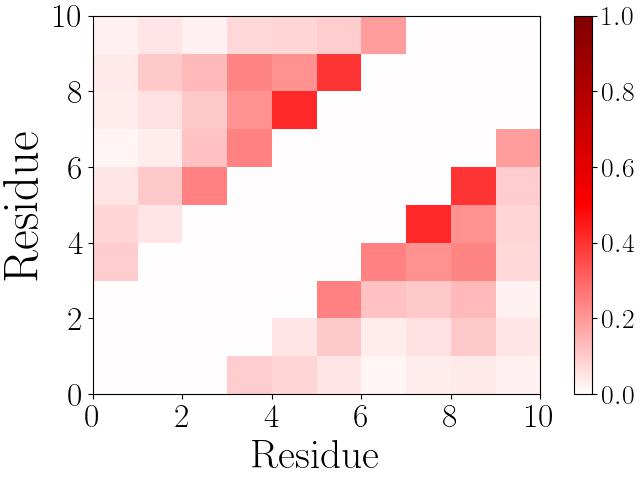}}
    \end{minipage}
    \caption{(a) \& (b) Folded and unfolded states of the Chignolin protein. (c) \& (d) Frequency of contacts between different residue pairs over all the identified folded and unfolded states, respectively.}
    \label{fig:chig_visuals}
\end{figure}

\subsubsection{Ornstein--Uhlenbeck process}

We simulate a one-dimensional Ornstein--Uhlenbeck process, given by
\begin{equation} \label{eq:SDE}
    dX_t = -\nabla V(X_t) \ts \mathrm{d}t + \sigma(X_t) \ts \mathrm{d}W_t,
\end{equation}
with $V(x) = \frac{\alpha}{2}x^2$ and $\sigma(X_t) = \sqrt{2 \beta^{-1}}$, where $\{W_t\}_{t \geq 0}$ is a Wiener process and $\beta>0$ the inverse temperature. To generate trajectory data $\{x_i, y_i\}_{i=1}^m$, we use the Euler--Maruyama scheme
\begin{equation*}
    X_{k+1} = X_k - \eta \nabla V(X_k) + 2 \beta^{-1} \Delta W_k,
\end{equation*}
where $\eta$ is the step size and $\Delta W_k = W_k - W_{k-1} \sim \mathcal{N}(0, \eta)$, i.e., normally distributed with mean $0$ and variance $\eta$. We choose $\alpha = 1$ and $ \beta = 4$ and generate $m = 5000$ data points with a lag time $\tau = 0.5$. The exact eigenvalues and eigenfunctions of the Koopman operator are
\begin{equation*}
    \lambda_i = e^{-\alpha(i-1)\tau}, \quad \varphi_i(x) = \frac{1}{\sqrt{(i-1)!}}H_{i-1}\big(\sqrt{\alpha \beta} x \big), \quad i = 1,2,3,\dots,
\end{equation*}
where $H_i$ denotes the $i$th probabilists' Hermite polynomial. We use the basis functions
\begin{align*}
    \mathcal{D} = \left\{ \psi_i(x, w_i) =  \exp\left(-\tfrac{(x - c_{i})^2}{2\sigma_{i}^2}\right), ~ i = 1, 2, \dots, n \right\},
\end{align*}
i.e., Gaussian functions with variable centers $ c_i $ and bandwidths $ \sigma_i $. That is, $ w_i = [c_i, \sigma_i]^\top $ and $ w \in \R^{2 \ts n} $ contains the centers and bandwidths of all basis functions. We select the $n=14$ basis functions shown in Figure~\ref{subfig:initial_basis_ou} to approximate the Koopman operator. Note that most of the centers of the Gaussian functions are initially in the left half of the domain. As a result, the approximation of the eigenfunctions will be poor in the right half. Applying the alternating optimization algorithm allows us to find more suitable basis functions. The optimized basis functions are shown in Figure \ref{subfig:final_basis_ou}. The centers are now evenly distributed and lead to a better approximation of the Koopman operator. The convergence of the reconstruction error \eqref{eq:param_edmd} is presented in Figure~\ref{subfig:errorK_ou} and the convergence of the VAMP-2 score in Figure~\ref{subfig:vamp_ou}. The eigenvalues and eigenfunctions of the approximated Koopman operator are presented in Figures~\ref{subfig:ou_eigv} and \ref{subfig:ou_eigf}, respectively. The numerically computed eigenvalues and eigenfunctions match the theoretical results.

\begin{figure}
  \centering
  \hspace{0.3cm}\subfloat[][\label{subfig:initial_basis_ou}]{\includegraphics[width=.44\textwidth]{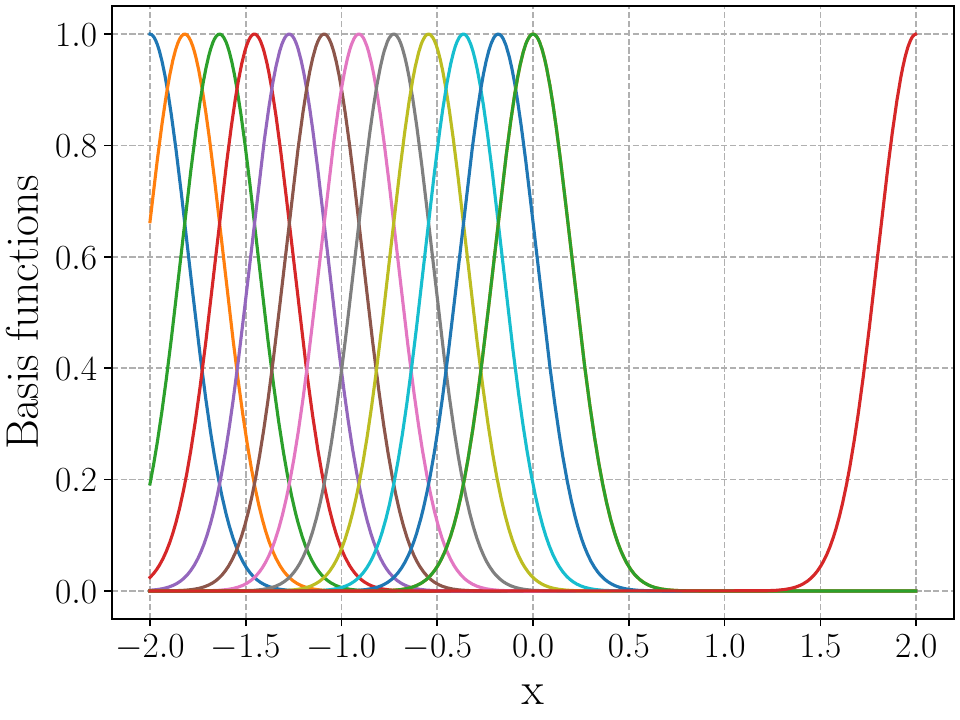}}\quad
  \subfloat[][\label{subfig:final_basis_ou}]{\includegraphics[width=.44\textwidth]{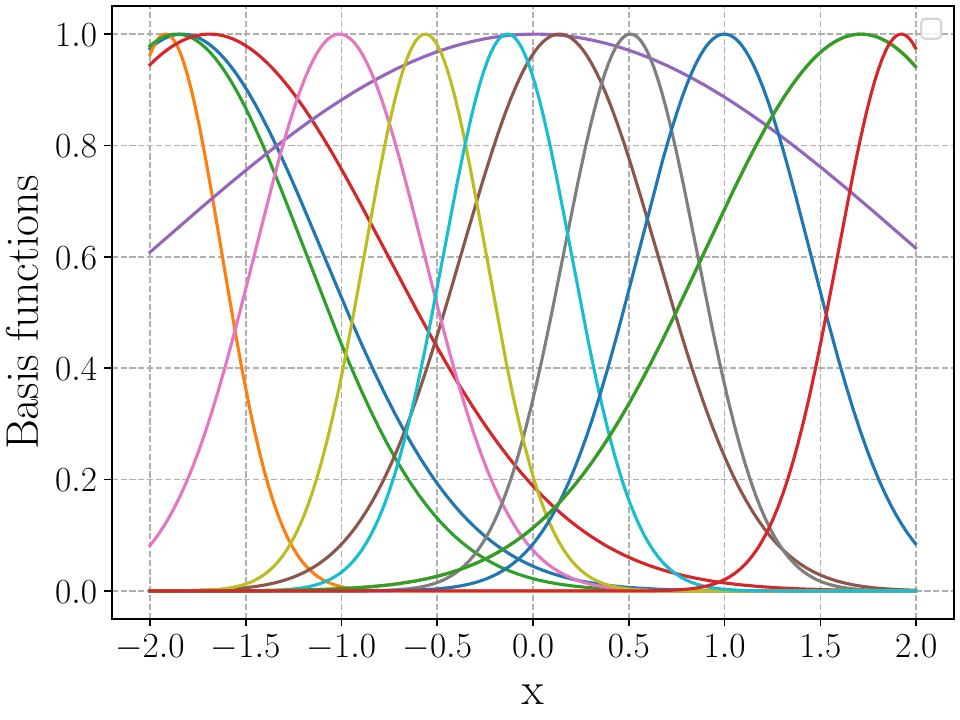}}\quad
  \subfloat[][\label{subfig:errorK_ou}]{\includegraphics[width=.44\textwidth]{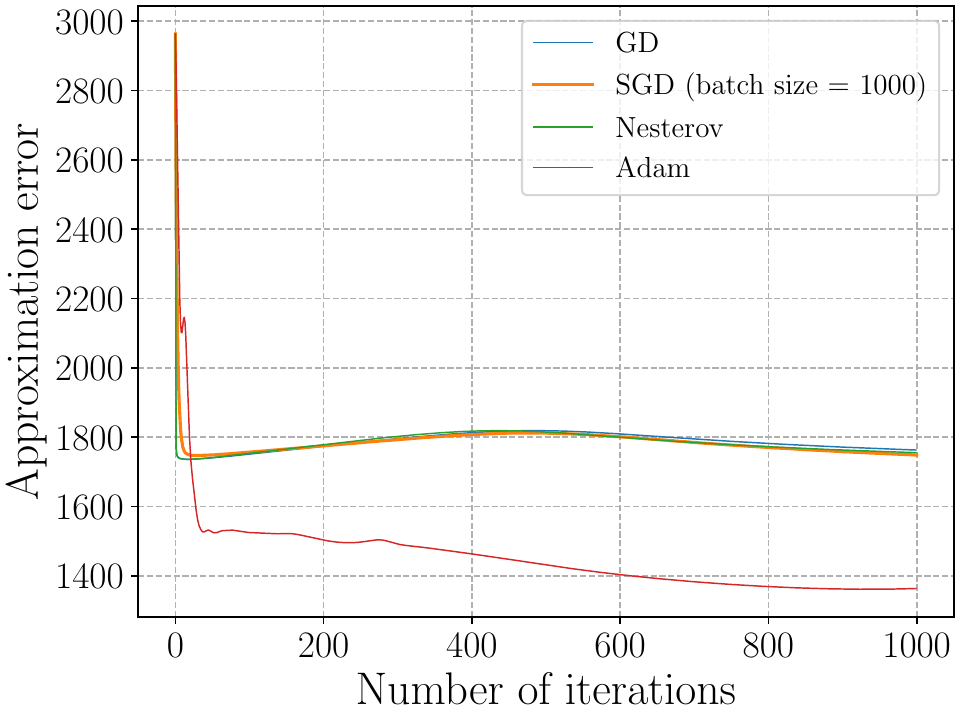}}\quad
  \subfloat[][\label{subfig:vamp_ou}]{\includegraphics[width=.44\textwidth]{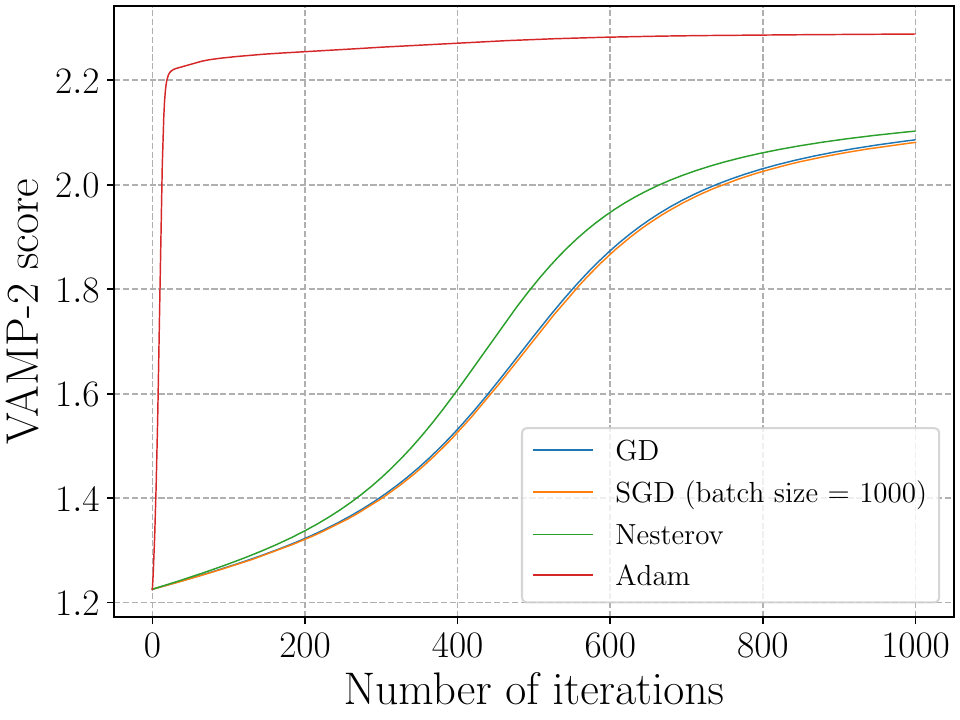}}\quad
  \subfloat[][\label{subfig:ou_eigv}]{\includegraphics[width=.44\textwidth]{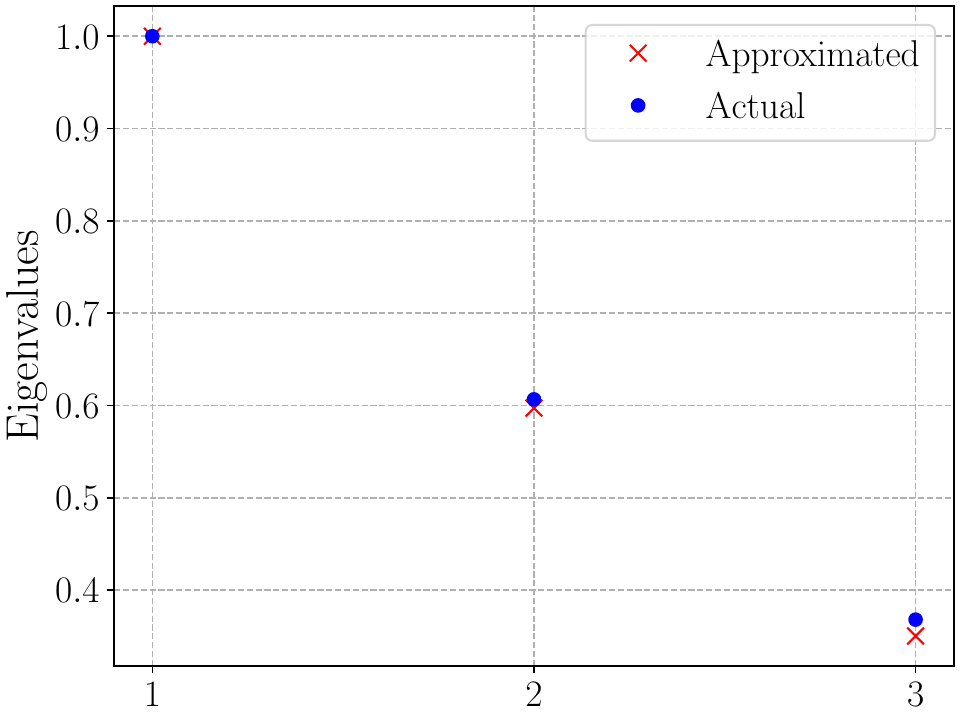}}\quad
  \subfloat[][\label{subfig:ou_eigf}]{\includegraphics[width=.46\textwidth]{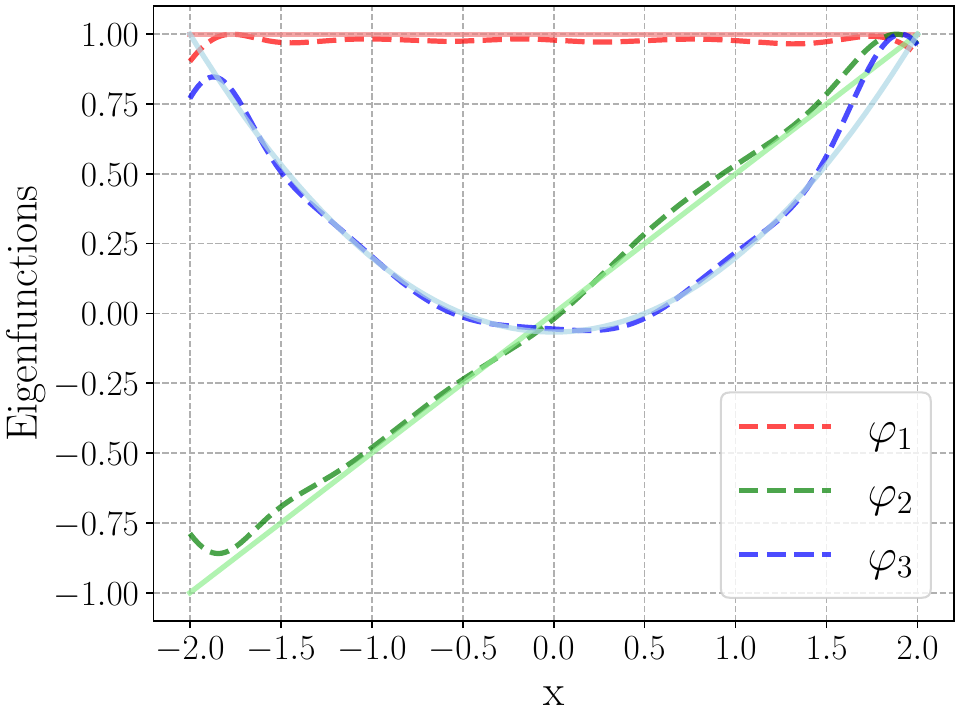}}\quad
  \caption{Koopman operator approximation for the Ornstein--Uhlenbeck process using parametric Gaussian functions. (a) Initial basis functions. (b) Optimized basis functions. (c) Convergence of the reconstruction error. (d) Convergence of the VAMP-2 score for optimizing the parameters of the basis functions. (e) Three dominant eigenvalues of the Koopman operator. (f) Corresponding eigenfunctions of the Koopman operator, where the solid lines represent the true eigenfunctions.}
  \label{fig:OU_results}
\end{figure}

\subsubsection{Triple-well 2D}

We now consider a two-dimensional triple-well problem \cite{schutte2013metastability}, given by
\begin{align*}
        V(x_1,x_2)&=3\exp({-x_1^2-(x_2-1/3)^2})-3\exp({-x_1^2-(x_2-5/3)^2}) \\ \nonumber
    & -5\exp(-(x_1-1)^2-x_2^2) - 5\exp(-(x_1+1)^2-x_2^2) \\ \nonumber
    &+\cfrac{2}{10}x_1^4 + \cfrac{2}{10}(x_2-1/3)^4.
\end{align*}
The potential is visualized in Figure \ref{subfig:tw_potential}. We set $\beta = 1.68$ and uniformly sample $m = 100\ts000$ training data points in the domain $[-2, 2] \times [-2, 2]$ and integrate the SDE \eqref{eq:SDE} using the Euler--Maruyama method. To approximate the Koopman operator, we use $n = 25$ Gaussian functions in the domain $[0,2] \times [-2,2]$ as shown in Figure~\ref{subfig:tw_initial_basis}. That is, the centers of all the basis functions are initially in the right half of the domain. The optimized basis functions are shown in Figure~\ref{subfig:tw_final_basis}. The centers moved and also the bandwidths changed significantly, leading to a more accurate approximation of the Koopman operator. The second eigenfunction of the approximated Koopman operator, separating the two deeper wells, is shown in Figure \ref{subfig:tw_eigf2}.

\begin{figure}
    \centering
    \begin{minipage}[t]{0.4\linewidth}
        \centering
        \subfloat[\label{subfig:tw_potential}]{%
        \includegraphics[width=0.9\linewidth]{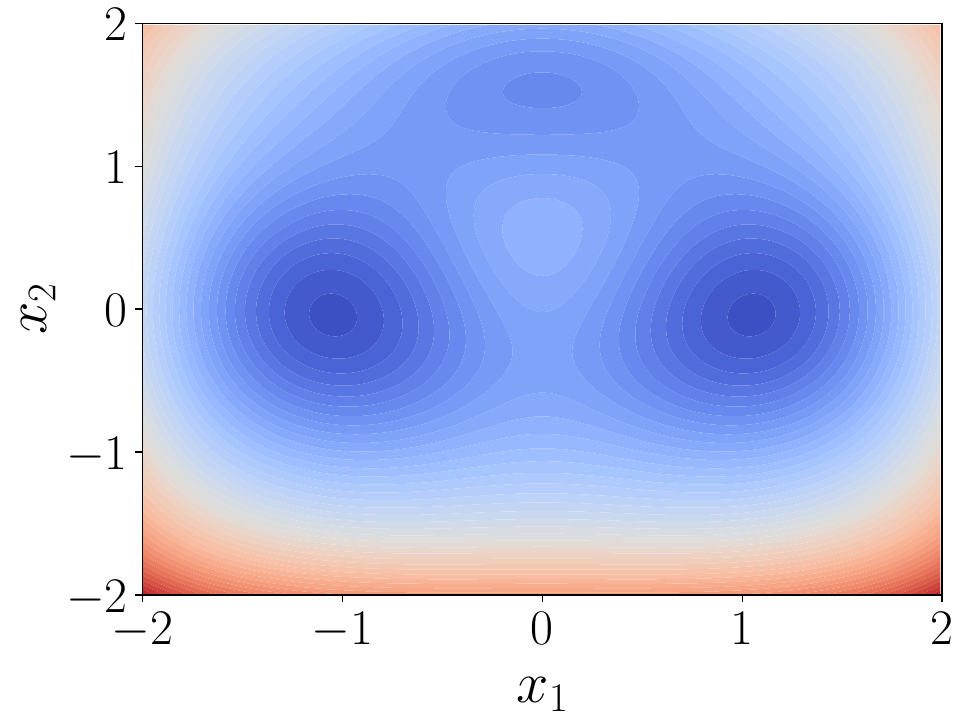}}
    \end{minipage}
    \begin{minipage}[t]{0.4\linewidth}
        \centering
        \subfloat[\label{subfig:tw_initial_basis}]{%
        \includegraphics[width=0.9\linewidth]{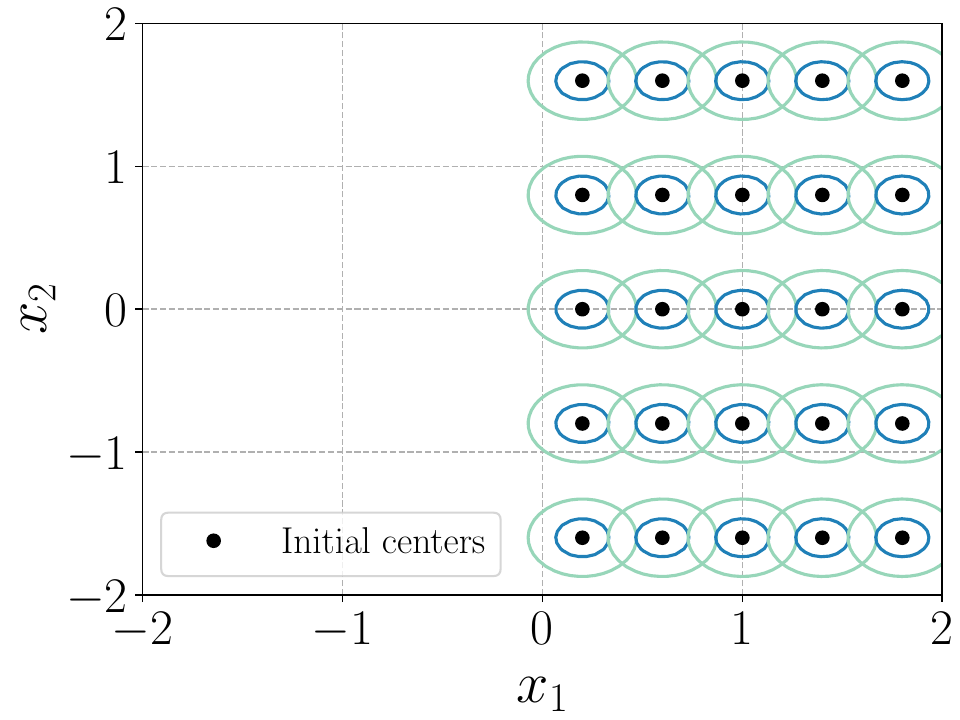}}
    \end{minipage} \\[1ex]
    \begin{minipage}[t]{0.4\linewidth}
        \centering
        \subfloat[\label{subfig:tw_final_basis}]{%
        \includegraphics[width=0.9\linewidth]{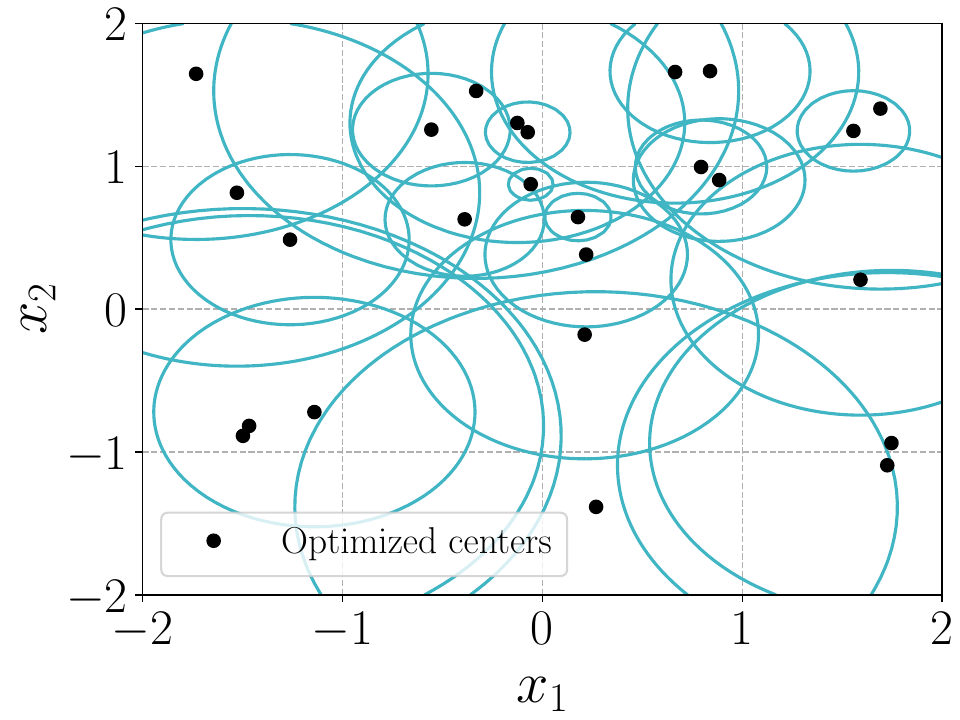}}
    \end{minipage}
    \begin{minipage}[t]{0.4\linewidth}
        \centering
        \subfloat[\label{subfig:tw_eigf2}]{%
        \includegraphics[width=0.7\linewidth]{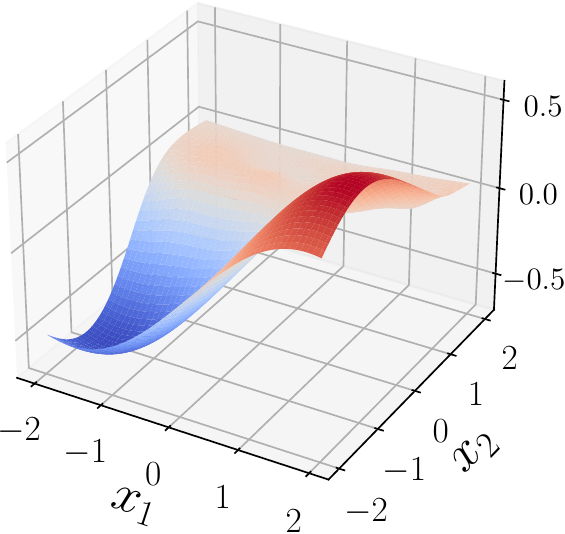}}
    \end{minipage}
    \caption{Koopman operator approximation for the triple-well problem using parametric Gaussian basis functions. (a) Potential $ V $. (b) Initial basis functions. (c) Optimized basis functions. (d) Second eigenfunction of the Koopman operator.}
    \label{fig:TW_results}
\end{figure}

\subsection{Parametric SINDy}

\subsubsection{Chua's circuit}

We consider the modified Chua circuit \cite{mishra2019modified} given by
\begin{equation*}
\begin{split}
    \dot{x}_1 &= \alpha[x_2 - f(x)], \\
    \dot{x}_2 &= x_1 - x_2 + x_3, \\
    \dot{x}_3 &= -\beta x_2,
\end{split}
\end{equation*}
where $f(x) = -b \sin\big(\frac{\pi x_1(t)}{a} + d\big)$ describes the electrical response of the resistor, see \cite{madan1993chua, kilicc2010practical}. We generate training data using $a = 2.6, b = 0.11, d = 0, \alpha = 10.2, \beta = 14.286$. To illustrate how the parameters affect the dynamics, we plot the reconstruction error for varying values of $w_1$ using the basis functions
\begin{equation*}
    \mathcal{D} = \left\{x_1, x_2, x_3, x_1 x_3, x_1 x_2, x_2^2, \sin(w_1 x_1), \cos(w_2 x_2) \right\}.
\end{equation*}
The results are shown in Figure \ref{subfig:modchua_energy}. The minimum reconstruction error is obtained for $w_1 \approx 1.208$, i.e., $a \approx 2.6$, which is the correct value of the parameter. We see that there are two regions in the energy landscape. Selecting the initial value for $w_1$ in Region 1, we obtain the global minimum, i.e., zero reconstruction error, whereas choosing the initial value to be in Region 2, the algorithm converges to the local minimum. The governing equations cannot be identified using standard SINDy, unless we explicitly choose the basis function $ \sin(\frac{\pi}{2.6} \ts x_1) $. We use the proposed framework with the above set of parametric basis functions $\mathcal{D}$ to find the optimal matrix $\Xi$ and parameters $w$. The convergence of the algorithms is illustrated in Figures~\ref{subfig:loss_modchua_psi} and \ref{subfig:loss_modchua_w}, respectively. The dynamics of the discovered equations shown in Figure \ref{subfig:modchua_gd_dyn} illustrate that the algorithm determines the correct matrix $\Xi$ and the correct value of the parameter $w_1$, provided that we start in Region 1.

\begin{figure}
\centering
    \begin{minipage}[t]{0.4\linewidth}
    \centering
    \subfloat[\label{subfig:modchua_energy}]{
        \includegraphics[width=.85\linewidth]{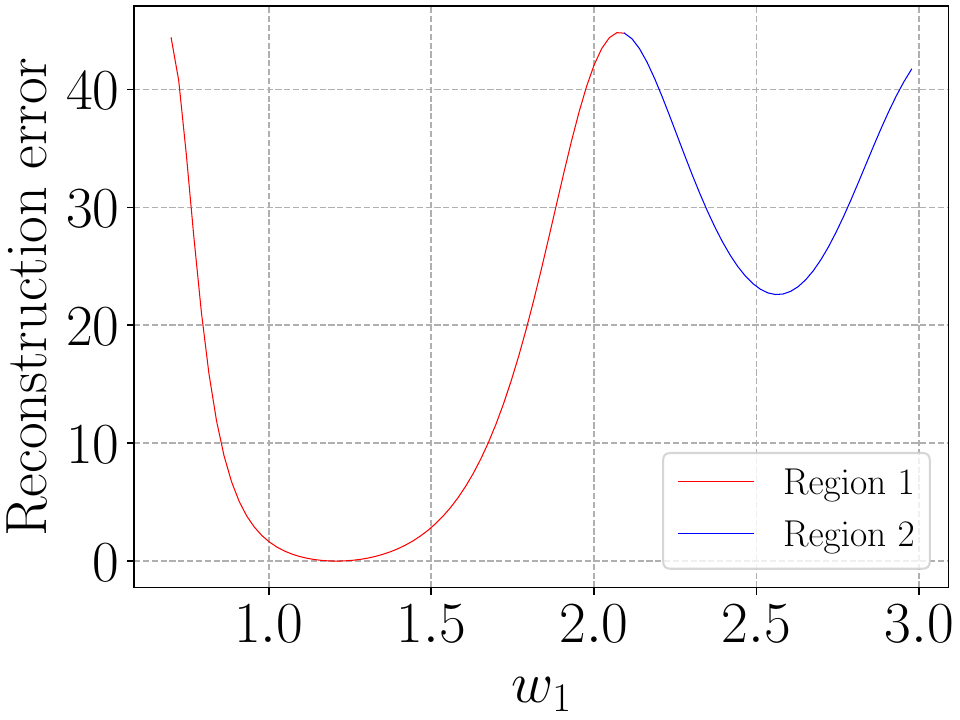}
    }
    \end{minipage}
    \begin{minipage}[t]{0.4\linewidth}
    \centering
    \subfloat[\label{subfig:loss_modchua_psi}]{
        \includegraphics[width=.85\linewidth]{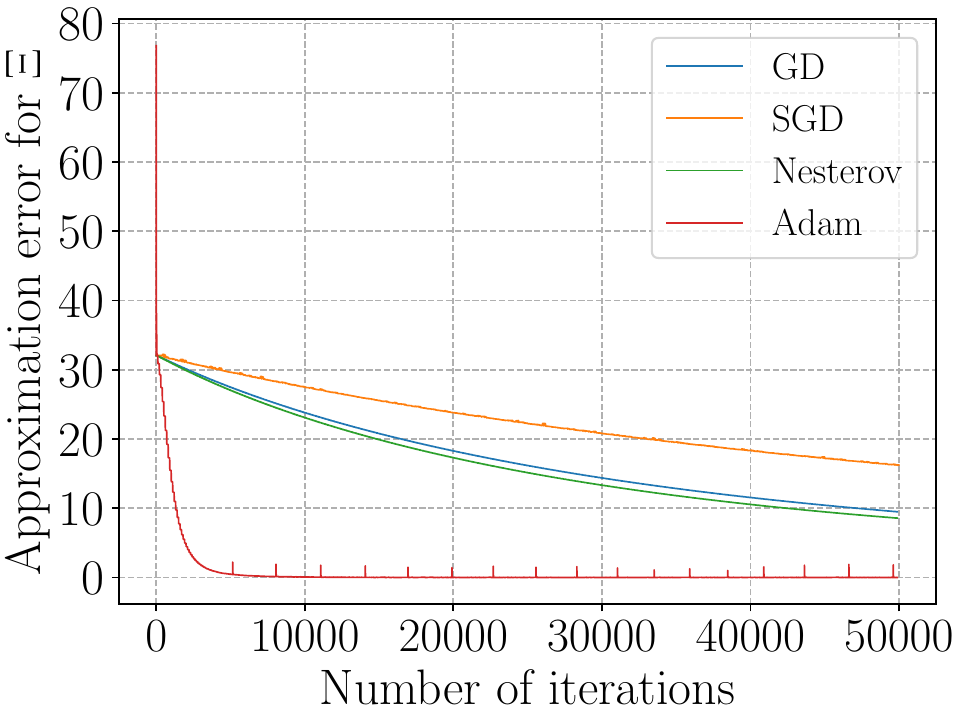}
    }
    \end{minipage}\\[1ex]
    \begin{minipage}[t]{0.4\linewidth}
    \centering
    \subfloat[\label{subfig:loss_modchua_w}]{
        \includegraphics[width=.85\linewidth]{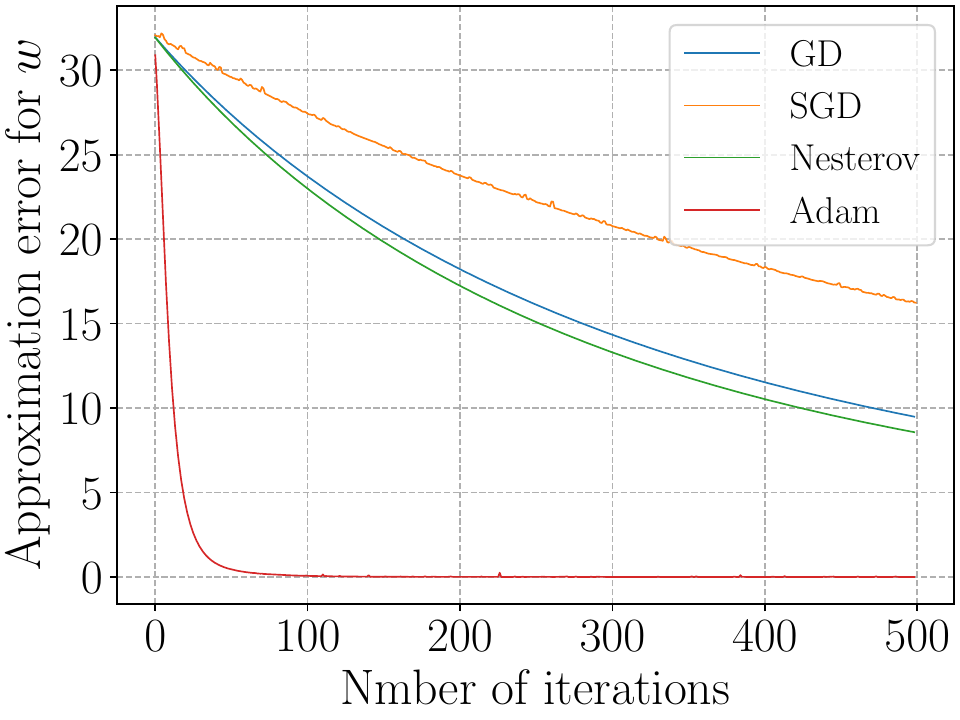}
    }
    \end{minipage}
    \begin{minipage}[t]{0.4\linewidth}
    \centering
    \subfloat[\label{subfig:modchua_gd_dyn}]{
        \includegraphics[width=0.9\linewidth]{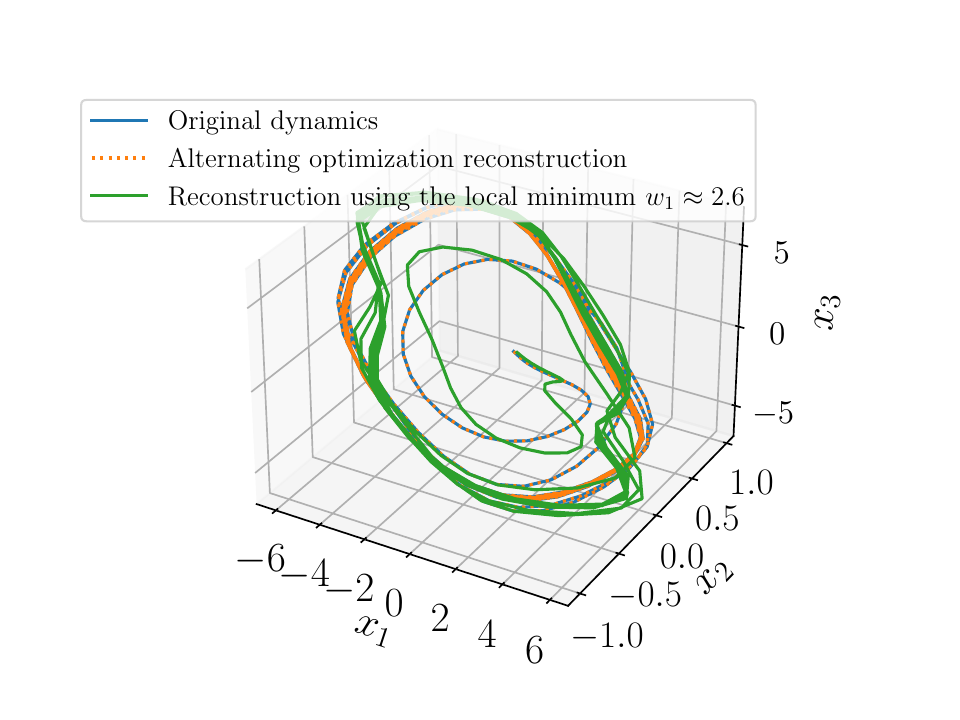}
    }
    \end{minipage}
    \caption{Results for Chua's circuit. (a) SINDy reconstruction error as a function of $ w_1 $. (b) \& (c) Convergence of the approximation errors for the loss functions for $\Xi$ and $w$, respectively. (d) Reconstruction of the dynamics using the optimization algorithms. The orange graph completely overlaps the blue, representing a perfect reconstruction of the dynamics. The green trajectory shows the resulting dynamics using a different value of $ w_1 $.}
    \label{fig:modchua_results}
\end{figure}

\subsubsection{Nonlinear heat PDE}

As a final example, we use the proposed framework to identify a PDE from simulation data. Consider the nonlinear heat equation
\begin{equation*}
    \rho c_p u_t = \frac{\partial}{\partial x}(\kappa(u) \ts u_x) = \frac{\partial}{\partial u}\kappa(u) \bigg(\frac{\partial u}{\partial x}\bigg)^2 + \kappa(u)\frac{\partial^2 u}{\partial x^2},
\end{equation*}
taken from \cite{filipov2018implicit}, where $u(x,t)$ is the temperature at position $x$ and time $t$, $c_p$ is the heat capacity, $\rho$ is the density, and $\kappa$ is the thermal conductivity that depends on the temperature~$u$. For $\kappa(u) = \kappa_0 e^{\chi u} $, we get
\begin{equation*}
    \rho c_p u_t = \kappa_0 \chi e^{\chi u} u_x^2 + \kappa_0 e^{\chi u} u_{xx}.
\end{equation*}
We choose $\rho = c_p = 1$, $\kappa_0 = 0.1$, and $\chi = -1$ and solve the equation using a finite difference scheme. For more details, see \cite{filipov2018implicit}. The temperature at the boundary is kept constant, i.e.,
\begin{equation*}
    u(1,t) = 2, \quad u(3,t) = 1, \quad t > 0,
\end{equation*}
and the initial temperature is
\begin{equation*}
    u(x, 0) = 2 - \frac{x - 1}{2} + (x - 1)(x - 3),
\end{equation*}
for $ x \in [1, 3] $. We then construct the matrix $\Theta(U(\chi))$ as discussed in Section~\ref{sec:background}, which includes different candidate terms to discover the PDE. Since the parameter $\chi$ is present within the exponential term, it cannot be identified using the standard PDE-FIND algorithm. We build a library of parametrized basis functions
\begin{align*}
    \Theta(U(\chi)) = \big [1, U, U_x, U U_x, U^2 U_x, U U_{xx}, U^2 U_{xx}, e^{\chi U}U_{x}^2, e^{\chi U} U_{xx}\big]
\end{align*}
and apply the alternating optimization algorithm to the reconstruction error \eqref{eq:param_pde_loss}. The resulting vector $\xi$ then contains the coefficients for the different terms. We obtain $\chi = -1.095$ and the PDE
\begin{equation*}
    u_t = -0.114 \ts e^{\chi u} u_x^2 + 0.105 \ts e^{\chi u} u_{xx}.
\end{equation*}
The small difference between the true PDE and the identified PDE is caused by the finite difference approximations of the derivatives. Nevertheless, the example shows that our method allows us to identify parameter-dependent PDEs from data.

\begin{figure}
    \centering
    \subfloat[][\label{subfig:pde_surface}]{\includegraphics[width=.27\textwidth]{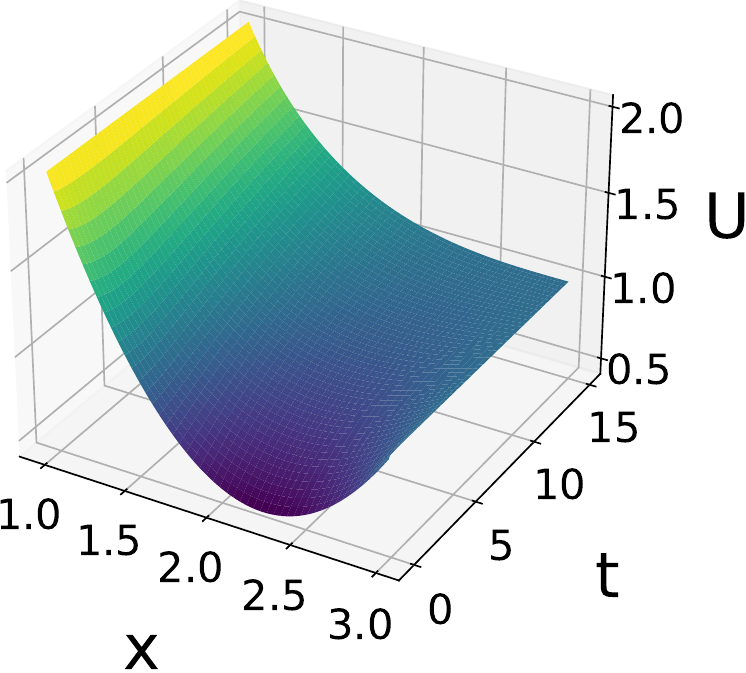}}\quad
    \subfloat[][\label{subfig:pde_error_contour}]{\includegraphics[width=.32\textwidth]{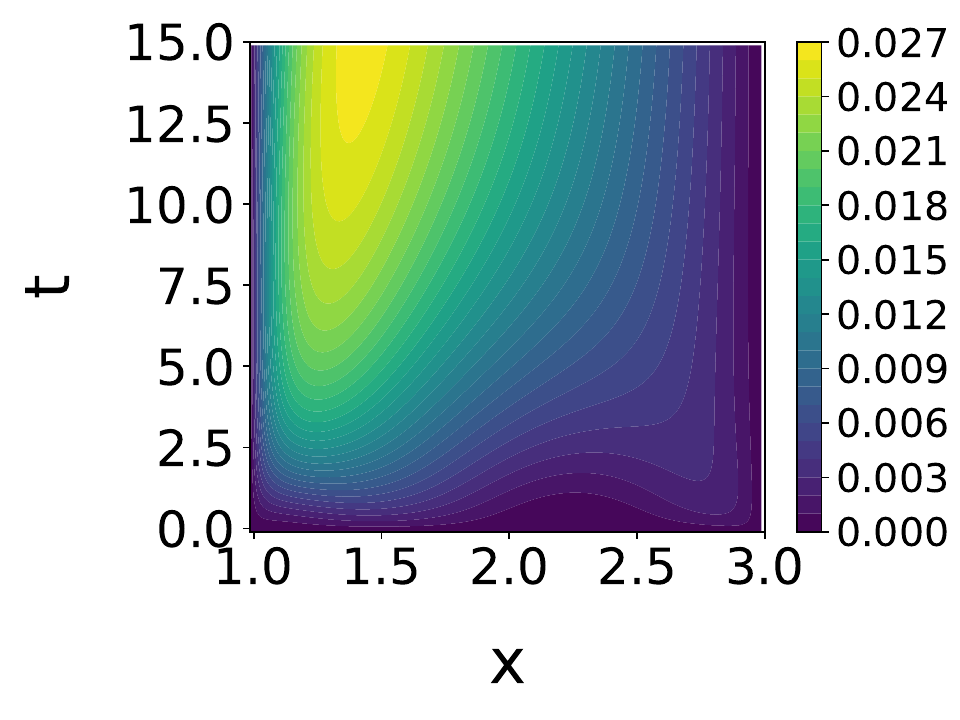}}\quad
    \vspace{-0.1cm}\subfloat[][\label{subfig:pde_adam_convergence}]{\includegraphics[width=.30\textwidth]{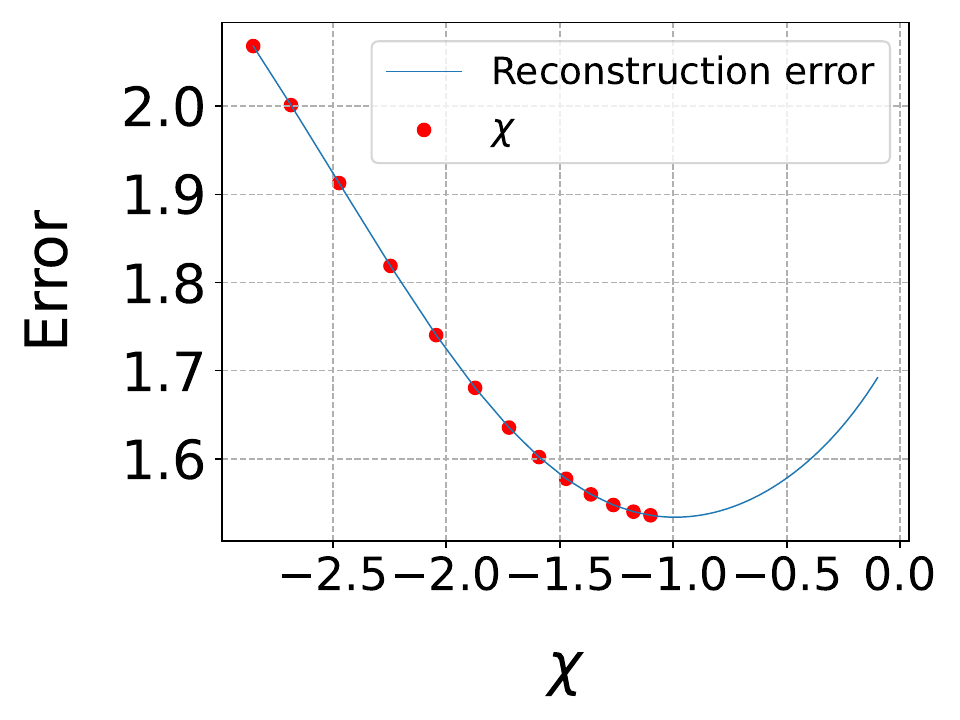}}\quad
    \caption{Simulation of the nonlinear heat equation. (a) Surface plot of the solution of the PDE. (b) Contour plot of the error between the original and the discovered PDE. (c)~Convergence of the algorithm to the correct value of $\chi$.}
    \label{fig:pde_results}
\end{figure}

\section{Conclusion and future work}
\label{sec:conclusion}

We proposed a novel data-driven alternating optimization framework for approximating the Koopman operator and discovering the governing equations of dynamical systems. The main goal was to learn optimal and interpretable dictionaries and the dynamics at the same time. We demonstrated the efficacy of the proposed framework using various benchmark problems such as the Ornstein--Uhlenbeck process, a triple-well potential, and protein folding data. Furthermore, we discovered the governing equations of Chua's circuit and a nonlinear heat equation with temperature-dependent thermal conductivity. The numerical results are promising and show that we can indeed learn more suitable dictionaries resulting in more accurate approximations of the Koopman operator or the dynamical system itself.

There are, however, open problems. One of the major challenges is selecting proper initial conditions and learning rates for the gradient descent algorithms. The algorithms might diverge or get stuck in local minima if the initial conditions or step sizes are chosen poorly. This leads to several interesting directions to extend this work in the future: A theoretical analysis of the optimization problems would help us understand the loss functions and the existence of local minima, which might allow us to choose suitable initial conditions and step sizes. An extension of this work would be to consider kernel-based variants of the data-driven algorithms, such as kernel EDMD \cite{WRK15,KSM20}. The proposed framework could then be used to optimize the kernel parameters.

\section*{Acknowledgments}
MT was supported by the EPSRC Centre for Doctoral Training in Mathematical Modeling, Analysis and Computation (MAC-MIGS) funded by the UK Engineering and Physical Sciences Research Council (grant EP/S023291/1), Heriot--Watt University and the University of Edinburgh. NKC is supported by an EPSRC-UKRI AI for Net Zero Grant: ``Enabling CO2 Capture and Storage Projects Using AI'', (Grant EP/Y006143/1). NKC is also supported by a City University of Hong Kong Start-up Grant, project number 7200809. We would like to thank D.E.\ Shaw Research for providing the Chignolin data.

\bibliographystyle{unsrturl}
\bibliography{references}

\end{document}